\documentclass[12pt]{amsart}   \usepackage{amssymb,latexsym}
\usepackage{amssymb,latexsym}
\usepackage{amsfonts,mathrsfs}
\usepackage[margin=1.4in]{geometry}
\usepackage{fancyhdr}

\newtheorem{theorem}{Theorem}[section]
\newtheorem{lemma}[theorem]{Lemma}
\newtheorem{proposition}[theorem]{Proposition}

\newtheorem{definition}[theorem]{Definition}

\theoremstyle{remark}
\newtheorem{remark}{Remark}[section]

\newcommand{\Vol}{\mathrm{Vol}}

\makeatletter

\newcommand{\Rmnum}[1]{\expandafter\@slowromancap\romannumeral #1@}
\makeatother
\begin{document}
\allowdisplaybreaks
\title{ Alexandrov-Fenchel type inequalities in the sphere}
\thanks{\noindent \textbf{MR(2010)Subject Classification}   53C23, 35J60, 53C42}
\author{Min Chen and Jun Sun}
\address{Min Chen, University of Science and Technology of China, No.96, JinZhai Road Baohe District,Hefei,Anhui, 230026,P.R.China.}
\email{cmcm@mail.ustc.edu.cn}
\address{Jun Sun, School of Mathematics and Statistics, Wuhan University, Wuhan, and Hubei Key Laboratory of Computational Science (Wuhan University), Wuhan, 430072, P. R. of China.}
\email{sunjun@whu.edu.cn}
\thanks{The authors are supported by the National Nature Science Foudation of China  No. 11721101 and National Key Research and Development Project No. SQ2020YFA070080. J. Sun is supported by NSFC 12071352}
\pagestyle{fancy}
\fancyhf{}
\renewcommand{\headrulewidth}{0pt}
\fancyhead[CE]{}
\fancyhead[CO]{\leftmark}
\fancyhead[LE,RO]{\thepage}

\begin{abstract}
In this paper, we attempt to use two types of flows to study the relations between quermassintegrals $\mathcal{A}_k$ (see Definition 1.1), which correspond to the Alexandrov-Fenchel inequalities for closed convex $C^2$-hypersurfaces in $\mathbb{S}_+^{n+1}.$
\end{abstract}

\maketitle
\numberwithin{equation}{section}
\section{Introduction}

The Alexandrov-Fenchel inequalities \cite{1,2} for the quermassintegrals of convex domains in $\mathbb{R}^{n+1}$ are fundamental in classical geometry. In \cite{3}, Guan and Li extended these inequalities to star-shaped domains in $\mathbb{R}^{n+1}$. There have been extensive interests on studying the Alexandrov-Fenchel type inequalities for quermassintegral in space forms. Let $N^{n+1}(K)$ be the space form with constant sectional curvature $K=1$, $0$ or $-1$. Under the Gaussian geodesic normal coordinates, the metric can be expressed as
\[ds^2=d\rho^2+\phi^2(\rho)dz^2,\]
where $\phi(\rho)=\sin\rho,\rho \in [0,\frac{\pi}{2})$ when $K=1$; $\phi(\rho)=\rho, \rho \in [0,\infty)$ when $K=0$; and $\phi(\rho)=\sinh\rho, \rho \in [0,\infty)$ when $K=-1,$ and $dz^2$ is the induced standard metric on ${\mathbb S}^n$ in Euclidean space.
We denote
\[\Phi(\rho)=\int^{\rho}_0\phi(s)ds\]
and consider the vector field $V=\phi(\rho)\frac{\partial}{\partial \rho}$. It is well known that $V$ is a conformal killing field. Let $M^n\subset \mathbb{N}^{n+1}(K)$ be a closed hypersurface and $\nu$ be the outward unit normal vector field. We call function $u=\langle V,\nu \rangle$ to be the generalized support function of the hypersurface.

Let $\kappa=\langle \kappa_1,\cdots, \kappa_n \rangle$ be the vector of $n$ principle curvatures of the hypersurface $M_0$ and we denote the $k$-th elementary symmetric function of $\kappa$ by $\sigma_k(\kappa).$ For space form $\mathbb{N}^{n+1},$ there exists a notion of quermassintegrals, which can be expressed as a family of curvature integrals using Cauchy-Cronfton formulas (e.g., \cite{17}).
\begin{definition}
Let $X:M\rightarrow \mathbb{N}^{n+1}(K)$ be a closed hypersurface embedded into $ \mathbb{N}^{n+1}(K)$ and $\sigma_k$ be the $k$-th elementary symmetric function of the second fundamental form. Suppose $\Omega$ is the domain enclosed by $M^n$ in $\mathbb{N}^{n+1}(K).$ The \textbf{$k$-th quermassintegal ${\mathcal A}_k$} is defined as follows:
\begin{align*}
&\mathcal{A}_{-1}(\Omega)=\Vol (\Omega),\\
&\mathcal {A}_0(\Omega)=\int_Md\mu_g,\\
&\mathcal {A}_1(\Omega)=\int_M\sigma_1d\mu_g+nK\Vol(\Omega),\\
&\mathcal {A}_k(\Omega)=\int_M\sigma_k d\mu_g+\frac{K(n-k+1)}{k-1}\mathcal{A}_{k-2}(\Omega),
\end{align*}
where $2\le k\le n.$
\end{definition}
In the hyperbolic space $\mathbb{H}^{n+1}$, there have been many interesting results. Brendle, Hung and Wang \cite{4} used the inverse mean curvature flow to obtain Minkowski type inequality with weighted factor (the relation between $\int_{M}\sigma_1\phi'$ and $\int_{\Omega}\phi'dvol$) for a compact mean convex hypersurface which is star-shaped with respect to the origin.
In the case of $h$-convexity, full range of quermassintegral inequalities were obtained in \cite{6,7} using expanding and contracting types of flows. Very recently, the results in \cite{7} for $h$-convex domains in $\mathbb{H}^{n+1}$ were reproved using flow (1.1)  directly in \cite{8} by establishing that $h$-convexity is preserved along flow (1.1). The sharp relation between $\mathcal{A}_2$ and $\mathcal{A}_0$ was previously proved in \cite{9} by a different method. Andrews, Chen and Wei \cite{21} replaced the $h$-convexity assumption with the weaker assumption of positive sectional curvature to obtain the relation between $\mathcal{A}_k$ and $\mathcal{A}_{-1}$ for $0\le k\le n-1.$
 In the general case of $k$-convexity, Brendle, Guan and Li \cite{5} considered the following parabolic evolution equation of a smooth one parameter family of embedded hypersurfaces $X(\cdot,t) \subset  \mathbb{N}^{n+1}(K),$
\begin{equation}
X_t=(\frac{\phi'(\rho)}{F}-\frac{u}{c_{n,k}})\nu,
\end{equation}
where $X$ is the position vector and $F:=\frac{\sigma_{k+1}(\kappa)}{\sigma_{k}(\kappa)}.$  The motivation to study Brendle-Guan-Li' flow is the Alexandrov-Fenchel type inequalities for quermassintegrals in space forms. If one establishes long time existence and convergence of the flow (1.1), the sharp Alexandrov-Fenchel inequalities for $k$-star-shaped domains in $\mathbb{N}^{n+1}(K)$ would follow. In the case $K=-1,$ the main problem is the preservation of star-shapedness along the flow. As a consequence, Brendle, Guan and Li \cite{5} established sharp inequalities between $\mathcal{A}_{n-1}$ and $\mathcal{A}_k$ for convex domains in $\mathbb{H}^{n+1}$, $\forall k\le n-1.$ For general $ 0\le k\le l\le n-1$ in  $\mathbb{H}^{n+1},$ with an extra initial gradient bound Condition $(\max_{x\in \mathbb{S}^n}|\nabla \ln(\cosh \rho)|^2(x,0)\le 12+3\min_{t=0} \sinh^2\rho)$, Brendle, Guan and Li \cite{5} can obtain the sharp inequalities of the relations between $\mathcal{A}_k$ and $\mathcal{A}_{l}$.

In $\mathbb{S}^{n+1},$ Gir$\tilde{a}o$ and Pinheiro \cite{10} used the inverse mean curvature flow to prove the Minkowski type inequality. Brendle, Guan and Li \cite{5} established sharp inequalities between $\mathcal{A}_{n-1}$ and $\mathcal{A}_k$ for convex domains in $\mathbb{S}^{n+1}$, $\forall k\le n-1.$
Wei and Xiong proved the following optimal inequalities for convex hypersurfaces in sphere.
\begin{theorem}
(see Theorem 1.3 in \cite{12})
Let $\Sigma^{n}$ be a closed and strictly convex hypersurface in $\mathbb{S}^{n+1}$. Then we have the following optimal inequalities $(k\le \frac{n}{2})$
\[\int_{\Sigma} L_kd\mu\ge C^{2k}_n(2k)!\omega_n^{\frac{2n}{k}}|\Sigma|^{\frac{n-2k}{n}}.\]
Equality holds if and only if $\Sigma$ is a geodesic sphere.
\end{theorem}
Here, the Gauss-Bonnet curvature is
\[L_k=C^{2k}_n(2k)!\sum^k_{i=0}\frac{C^i_k}{C^{2k-2i}_n}(-1)^i\sigma_{2k-2i}.\]

Makowski and Scheuer proved the following Alexandrov-Fenchel type inequalities  in the spheres. Namely
\begin{theorem}
(see Theorem 7.6 in \cite{11} )
Let $M\subset \mathbb{S}^{n+1}$ be an embedded, closed connected and convex $C^2$-hypersurface of the sphere. Let $k\in \mathbb{N}_+$ with $2k+1\le n$ and let $\hat{M}$ be convex body enclosed by $M$. Then we have the inequality
\[W_{2k+1}(\hat{M})\ge \frac{\omega_n}{n+1}\sum^k_{i=0}(-1)^i\frac{n-2k}{n-2k+2i}\begin{pmatrix} k\\ i\end{pmatrix} (\frac{n+1}{\omega_n}W_1(\hat{M}))^{\frac{n-2k+2i}{n}},\]
and equality holds if and only if $M$ is a geodesic sphere.
\end{theorem}

By Proposition 7 in \cite{20}, we have
\begin{equation}\label{Wk}
W_{k+1}=\frac{1}{(n+1)\begin{pmatrix} n\\ k\end{pmatrix}}\mathcal{A}_k.
\end{equation}
In view of  $(\ref{Wk}),$ Theorem 1.3 implies the relation between $A_{2k}$ and $A_0$ in $\mathbb{S}^{n+1}$ (Ge, Wang and Wu also proved a similar type of inequalities in $\mathbb{H}^{n+1}$ in Theorem 1.3 in \cite{6} before).
The optimal inequalities between two quermasssintegrals $\mathcal{A}_k$ and $\mathcal{A}_l$ (for general $ 0\le k\le l\le n-1$) in $\mathbb{S}^{n+1}$ are still open.  We will derive the relation between two adjacent quermassintegrals for general even number (i.e., $A_{2k}$ and $A_{2k-2}$). We will also derive the relation between two adjacent quermassintegrals for general odd number (i.e., $A_{2k+1}$ and $A_{2k-1}$) as well.

 Assume that $s_k=\mathcal{A}_k(B_{\frac{\pi}{2}}(o)) ,$ we now state the main result of this paper:

\begin{theorem} \label{thm:1.2}
Let $M$ be a closed convex $C^2$-hypersurface in $\mathbb{S}_+^{n+1},$ then the following inequality holds,
\begin{equation}\label{def-Ak-Ak-2}
\mathcal{A}_k(\Omega)\ge \xi_{k,k-2}(\mathcal{A}_{k-2}(\Omega) )\qquad \text{for any}  \quad 1\le k\le n-1,
\end{equation}
where $\xi_{k,k-2}$ is the unique positive function defined on $(0,s_{k-2})$ such that the equality holds if and only if $M$ is a geodesic sphere.
\end{theorem}

We will prove the relations between $\mathcal{A}_{k}$ and $\mathcal{A}_{k-2}$ by induction. It is already known that hypersurfaces converge smoothly to the equator along the Gerhadt's flow (see flow (4.1)). In Section 4, under the assumption that $\mathcal{A}_{k-2}(\Omega)\ge \xi_{k-2,k-4}(\mathcal{A}_{k-4}(\Omega))$ (the equality holds if and only if $M$ is a geodesic sphere), we could apply Gerhardt's flow to prove that $\mathcal{A}_{k}(\Omega)\ge \xi_{k,k-2}(\mathcal{A}_{k-2}(\Omega))$(the equality holds if and only if $M$ is a geodesic sphere).  We can also use Gerhardt's flow to derive the relation between $\mathcal{A}_{2}$ and $\mathcal{A}_{0}$ as well. However, we are unable to use Gerhardt's flow to obtain the relation between $\mathcal{A}_{1}$ and $\mathcal{A}_{-1}$ as we are failed  to get enough information about $\xi_{1,-1}$ as we describe in (\ref{xi_{k,k-2}}) (see also (\ref{xi_{2,0}})) by the definition of $\xi_{k,k-2}$.  To derive the relation of $\mathcal{A}_1$ and $\mathcal{A}_{-1}$, we will use Chen-Guan-Li-Scheuer's flow (see flow (\ref{flow-CGLS})) in Section 3. A nice feature of this type of flow is that we do not need to know specific information about $\xi_{1,-1}$ and we just need to obtain the monotonicity property of $\mathcal{A}_{1}$ and $\mathcal{A}_{-1}$ along the flow to acquire the relation between $\mathcal{A}_{1}$ and $\mathcal{A}_{-1}$. More generally, we can prove the relation between $\mathcal{A}_m$ and $\mathcal{A}_{-1}$ for $0\le m\le n-1$.  Except in the case of $k=0$ (i.e., flow (\ref{k=0})) where $C^2$ estimates follow directly from the theory of quasi-linear PDE, $C^2$ estimates for solutions of flow $(\ref{flow-CGLS})$ is still an open question. As a result, combining applying Gerhardt's flow with Chen-Guan-Li-Scheuer's flow together, we could establish a full range of  Alexandrov-Fenchel inequalities described in Theorem \ref{thm:1.2}.

The subsequent sections of this paper are organized as follows: in Section 2, we will recall some general facts on the elementary symmetric functions and the quermassintegrals; in Section 3, we will use Chen-Guan-Li-Scheuer's flow to prove the the relation between ${\mathcal A}_m$ and ${\mathcal A}_{-1}$; in Section 4, we will use Gerhardt's flow to finish the proof of the main theorem.

\vspace{.1in}

\section{Setting and general facts}
Let us present some basic facts which will be used later in this paper.
\begin{definition}(\cite{18})
For $1\le k\le n,$ let $\Gamma_k$ be a cone in $\mathbb{R}^n$ determined by
\[\Gamma_k=\{\lambda \in \mathbb{R}^n: \sigma_1>0,\cdots, \sigma_k>0\}.\]
An $n\times n$ symmetric matrix $W$ is called belonging to $\Gamma_k$ if $\lambda(W)\in \Gamma_k.$
\end{definition}
Then we will introduce Newton-Maclaurin inequalities and Minkowski identity.
\begin{lemma}\label{lemma-NM}(\cite{18})
For $W\in \Gamma_k,$
\[(n-k+1)(k+1)\sigma_{k-1}(W)\sigma_{k+1}(W)\le k(n-k)\sigma_k^2(W),\]
and
\[\sigma_{k+1}(W)\le c_{n,k}\sigma_k^{\frac{k+1}{k}}(W),\]
where $c_{n,k}=\frac{\sigma_{k}^{\frac{k+1}{k}}}{\sigma_{k+1}}(I).$ The equality holds if and only if $W=cI$ for some $c>0.$
\end{lemma}

\begin{proposition}\label{prop-Min}(\cite{5})
Let $M$ be a closed hypersurfaces in $\mathbb{N}^{n+1}(K).$ Then, for $k=0,1,\cdots,n-1,$
\[(k+1)\int_M\sigma_{k+1}u=(n-k)\int_M\phi'(\rho)\sigma_k(\kappa).\]
We will use the convention that $\sigma_0\equiv 1.$
\end{proposition}

We define
\begin{equation}\label{Phi}
\Phi(\rho)=\int^{\rho}_{0}\phi(r)dr.
\end{equation}
Then $\Phi(\rho)$ is $\frac{\rho^2}{2}, \cosh \rho-1, 1-\cos \rho$ for $K=0,-1$ and 1, respectively.
The following lemma holds for general warped product manifolds.
\begin{lemma}(\cite{15})
Let $M^n\subset N^{n+1}$ be a closed hypersurface with induced metric $g$. Let $\Phi$ be defined as in (\ref{Phi}) and $V=\phi(\rho)\frac{\partial}{\partial \rho}$. Then $\Phi |_M$ satisfies,
\[\nabla_i\nabla_j \Phi=\phi'(\rho)g_{ij}-h_{ij}\langle V,\nu \rangle,\]
where $\nabla$ is the covariant derivative with respect to $g,$ $\nu$ is the outward unit normal, and $h_{ij}$ is the second fundamental form of the hypersurface.
\end{lemma}

We also recall the gradient and hessian of the support function $u:=\langle V,\nu\rangle$ under the induced metric $g$ on $M$.
\begin{lemma}(\cite{15})
The support function $u$ satisfies
\begin{align*}
&\nabla_i u=h_{il}\nabla_l\Phi;\\
&\nabla_i\nabla_j u=\nabla_l h_{ij}\nabla_l \Phi+\phi' h_{ij}-u(h^2)_{ij},
\end{align*}
where $(h^2)_{ij}:=g^{kl}h_{ik}h_{jl}.$
\end{lemma}

Next we will present the evolution equations of $\sigma_l$ and quermassintegrals.
Let $M(t)$ be a smooth family of closed hyersurfaces in $\mathbb{N}^{n+1}(K).$ Let $X(\cdot,t)$ denote a point on $M(t).$  We will consider the flow
\begin{equation}\label{f}
X_t=f\nu.
\end{equation}
Along this flow, we have

\begin{proposition} \label{evolu equa} (\cite{5})
Under the flow (\ref{f}) for the hypersurface in a Riemannian manifold, suppose $\Omega$ is the domain enclosed by the closed hypersurface, we have the following evolution equations.
\begin{align*}
\partial_t g_{ij}&=2fh_{ij}\\
\partial_t h_{ij}&=-\nabla_i\nabla_j f+f(h^2)_{ij}-fR_{\nu ij \nu}\\
\partial_t h^j_i&=-g^{jk}\nabla_k\nabla_j f-g^{jk}f(h^2)_{ki}-fg^{jk}R_{\nu ik \nu}\\
\partial_t \sigma_k& =\frac{\partial \sigma_k}{\partial h^j_i}\partial_t h^j_i
\end{align*}
Moreover, if $N$ has constant sectional curvature $K$, then for $l\ge 0,$ we have
\begin{equation}\label{evolution of sigma}
\partial_t \int_M\sigma_l=\int_M f[(l+1)\sigma_{l+1}-(n-l+1)K\sigma_{l-1}]d\mu_g,
\end{equation}
and
\[\partial_t \Vol(\Omega)=\int_M fd\mu_g.\]
\end{proposition}

Using Proposition \ref{evolu equa}, we have the following proposition, which motivates the definition of the quermassintegrals:

\begin{proposition} (\cite{5}) In $\mathbb{N}^{n+1}(K),$ along the flow (\ref{f}) for $0\le l<n-1,$ we have
\begin{equation}\label{evolu equa A_k}
\partial_t \mathcal{A}_l=(l+1)\int_Mf\sigma_{l+1}.
\end{equation}
\end{proposition}

Let $B_{\rho}(o)\subset \mathbb{N}^{n+1}(K)$ be the geodesic ball of radius $\rho$ centered at the origin $o$. Then we have
\[\frac{d}{d\rho}(\mathcal{A}_k(B_{\rho}(o)))=(k+1)\int_{\partial B_{\rho}(o)}\sigma_{k+1}>0,\]
for any $k=-1,0,1,2,\cdots,n-1$. If we view $\mathcal{A}_k(B_{\rho}(o))$ as a function of $\rho$, then the inverse function can be denoted as
\[\rho=\eta_k(\mathcal{A}_k(B_{\rho}(o))),\]
where $\eta_k : (0,s_k)\rightarrow (0,\pi/2)$ is a strictly increasing function for any fixed $k$.   Brendle, Guan and Li \cite{5} want to compare the relation between $\mathcal{A}_k$ and $\mathcal{A}_{l}$ for given $k>l$ and balls $B_{\rho}(o), \rho>0$ in $\mathbb{N}^{n+1}(K),$  which are the optimal solutions of this isoperimetric problem. Let $\xi_{k,l}$ be the unique positive strictly increasing function defined on $(0,s_l)$ such that
\begin{equation}\label{xi_{k,l}}
\mathcal{A}_{k}(B_{\rho}(o))=\xi_{k,l}(\mathcal{A}_{l}(B_{\rho}(o))).
\end{equation}
In general, for a bounded domain $\Omega \subset \mathbb{N}^{n+1}(K),$  they want to establish
\begin{equation}
\mathcal{A}_{k}(\Omega)\ge \xi_{k,l}(\mathcal{A}_{l}(\Omega)).
\end{equation}
In this following part of this paper, we will prove the relation between $\mathcal{A}_{k}$ and $\mathcal{A}_{k-2}$. More precisely, we will show that
 \[\mathcal{A}_{k}(\Omega)\ge \xi_{k,k-2}(\mathcal{A}_{k-2}(\Omega)).\]

 \vspace{.1in}

\section{Chen-Guan-Li-Scheuer's flow and applications}
In this section, we will establish the relation between $\mathcal{A}_1(\Omega)$ and $\mathcal{A}_{-1}(\Omega)$. Actually, we will provide the relationship between $\mathcal{A}_m$ and $\mathcal{A}_{-1}$ for $1\le m\le n-1.$
In \cite{13}, Chen, Guan, Li and Scheuer introduced the following flow
\begin{equation}\label{flow-CGLS}
X_t=(c_{n,k}\phi'-\frac{\sigma_{k+1}}{\sigma_k}u)\nu,
\end{equation}
Similar to the Brendle-Guan-Li's flow, monotonicity property for quermassintegrals holds as long as the flow (\ref{flow-CGLS}) exists. It is claimed in \cite{16} and proved in \cite{13} that flow (\ref{flow-CGLS}) preserves convexity, and one has $C^1$-estimates for solutions, and upper and lower bounds for $F=\frac{\sigma_{k+1}}{\sigma_k}$ along the flow (\ref{flow-CGLS}).
For completeness, we will include a proof of convexity preserving which we will use later here.

\subsection{Convexity preserving}

In this subsection, we will prove that if $M$ is strictly convex, then along the normalized flow (\ref{flow-CGLS}), the principle curvatures of $M_t$ remain strictly positive. We will need the following algebraic lemma.

\begin{lemma}\label{convexity property}(\cite{19})
Let $F=\frac{\sigma_{k+1}}{\sigma_k}(h_{ij}),$  and $\{\tilde{h}^{ij}\}$ be the inverse matrix of $h_{ij},$ then
\[(F^{ij,rs}+2F^{ir}\tilde h^{js})\eta_{ij}\eta_{rs}\ge 2F^{-1}(F^{ij}\eta_{ij})^2,\]
for any real symmetric $n\times n$ matrix $(\eta_{ij})$.
Here, $F^{ij}=\frac{\partial F}{\partial h_{ij}}$ and $F^{ij,rs}=\frac{\partial^2 F}{\partial h_{ij}\partial h_{rs}}$.
\end{lemma}

\begin{lemma}\label{convexity preserving}
Let $X(\cdot,t)$ be a smooth, closed and strictly convex solution to the normalized flow (\ref{flow-CGLS}) for $t\in [0,T),$ which encloses the origin.  There is a positive constant $C$ depending only on $M$, the upper and lower bounds of $\rho$ ($C^0$ estimate), the gradient estimate $\rho$ ($C^1$ estimate), and the upper bound of $F$ such that the principal curvatures of $X(\cdot,t)$ are bounded from below
\[\kappa_i(\cdot,t)\ge \frac{1}{C}, \quad \forall t\in [0,T) \quad \text{and} \quad i=1,\cdots,n.\]
\end{lemma}

\begin{proof} From Proposition \ref{evolu equa} and using the fact that $\nabla_i\nabla_j\phi'=-K\nabla_i\nabla_j\Phi$, we compute
\begin{align*}
\partial_th_{ij}&=-\nabla_i\nabla_j(c_{n,k}\phi'-uF)+(c_{n,k}\phi'-uF)(h^2)_{ij}-K(c_{n,k}\phi'-uF)\delta^i_j\\
&=-(c_{n,k}\nabla_i\nabla_j\phi'-\nabla_i\nabla_j uF-\nabla_iu\nabla_jF-\nabla_iF\nabla_ju-u\nabla_i\nabla_jF)\\
&+(c_{n,k}\phi'-uF)(h^2)_{ij}-K(c_{n,k}\phi'-uF)\delta^i_j\\
&=Kc_{n,k}\nabla_i\nabla_j\Phi+\nabla_i\nabla_juF+\nabla_iu\nabla_jF+\nabla_iF\nabla_ju+u\nabla_i\nabla_jF\\
&-(c_{n,k}\phi'-uF)(h^2)_{ij}-K(c_{n,k}\phi'-uF)\delta^i_j\\
&=Kc_{n,k}(\phi'g_{ij}-uh_{ij})+(\nabla h_{ij}\nabla\Phi+\phi'h_{ij}-u(h^2)_{ij})F+h_{il}\nabla_l\Phi\nabla_jF+h_{jl}\nabla_{l}\Phi\nabla_{i}F\\
&+u\big(F^{kl}\nabla_k\nabla_lh_{ij}+(F^{kl}(h^2)_{kl}-KF^{kk})h_{ij}-F((h^2)_{ij}-Kg_{ij})+F^{kl,pq}\nabla_ih_{pq}\nabla_jh_{kl}\big)\\
&+(c_{n,k}\phi'-uF)(h^2)_{ij}-K(c_{n,k}\phi'-uF)\delta^i_j.\\
\end{align*}
Thus
\begin{align*}
\partial_th_{ij}=uF^{kl}\nabla_k\nabla_lh_{ij}+uF^{kl,pq}\nabla_ih_{pq}\nabla_j h_{kl}+\nabla h_{ij}\nabla\Phi F+h_{il}\nabla_l \Phi \nabla_jF+h_{jl}\nabla_l \Phi \nabla_iF\\
+(c_{n,k}\phi'-3uF)(h^2)_{ij}+\big((F^{kl}(h^2)_{kl}-KF^{kk})u-c_{n,k}Ku+\phi'F\big)h_{ij}+2KuF\delta^i_j.
\end{align*}
Assume $(\tilde h^{ij})$ to be the inverse matrix of $(h_{ij})$. The principle radii of curvature of $M_t$ are the eigenvalues of $\{\tilde h^{ik} g_{kj}\}.$ To derive a positive lower bound of principle curvatures, it suffices to prove that the eigenvalues of $(\tilde h^{ik}g_{kj})$ are bounded from above. For this, we consider the following quantity
\[W(x,t)=\log \Lambda(x,t)-\log u(x,t),\]
where
\[\Lambda(x,t)=\max\{\tilde h^{ij}(x,t)\xi_i\xi_j:g^{ij}(x,t)\xi_i\xi_j=1\}.\]
Assume $W$ attains its maximum on $\mathbb S^n\times [0,T']$ at $(x_0,t_0)$  for any fixed $T'<T$.
We choose a local orthonormal frame $e_1,e_2,\cdots,e_n$ on $M_t$ such that $(h_{ij})$ is diagonal at $F(x_0,t_0)$. By a rotation, we may also suppose that $\Lambda(x_0,t_0)=\tilde h^{ij}(x,t)\xi_i\xi_j$ with $\xi=(1,0,\cdots,0).$
Let
\[\omega(x,t)=\log \lambda(x,t)-\log u(x,t),\]
where $\lambda(x,t)=\tilde h^{11}/g^{11}.$ Then $\max_{\mathbb S^n\times [0,T']}W=\max_{\mathbb S^n\times [0,T']}\omega$ and so $\omega$ achieves its maximum at $(x_0,t_0).$ In the following we prove an upper bound for $\omega.$
\begin{align*}
\partial_t \lambda=&-(\tilde h^{11})^2\partial_th_{11}+\tilde h^{11}\partial_t g_{11}\\
=&-(\tilde h^{11})^2\Big(uF^{kl}\nabla_k\nabla_lh_{11}+uF^{kl,pq}\nabla_1h_{pq}\nabla_1 h_{kl}+\nabla h_{11}\nabla\Phi F+h_{1l}\nabla_l \Phi \nabla_1F\\
&+h_{1l}\nabla_l \Phi \nabla_1F+(c_{n,k}\phi'-3uF)(h^2)_{11}\\
&+\big((F^{kl}(h^2)_{kl}-KF^{kk})u-c_{n,k}Ku+\phi'F\big)h_{11}+2KuF\Big)+2(c_{n,k}\phi'-uF)\\
=&-(\tilde h^{11})^2\big(uF^{kl}\nabla_k\nabla_lh_{11}+uF^{kl,pq}\nabla_1h_{pq}\nabla_1 h_{kl}\big)+\nabla \lambda \nabla\Phi F-2\tilde h^{11}\nabla_1 \Phi \nabla_1F\\
&-(c_{n,k}\phi'-3uF)-\tilde h^{11}\big((F^{kl}(h^2)_{kl}-KF^{kk})u-c_{n,k}Ku+\phi'F\big)\\
&-2KuF(\tilde h^{11})^2+2(c_{n,k}\phi'-uF)\\
=&uF^{kl}\nabla_k\nabla_l \lambda-2uF^{kl}(\tilde h^{11})^2\tilde h^{pq}\nabla_1h_{kp}\nabla_1h_{lq}-u(\tilde h^{11})^2F^{kl,pq}\nabla_1h_{pq}\nabla_1h_{kl}\\
&+\nabla \lambda \nabla\Phi F-2\tilde h^{11}\nabla_1 \Phi \nabla_1F+(c_{n,k}\phi'+uF)\\
&-\tilde h^{11}\big((F^{kl}(h^2)_{kl}-KF^{kk})u-c_{n,k}Ku+\phi'F\big)-2KuF(\tilde h^{11})^2,
\end{align*}
where we used the fact that
\begin{align*}
\nabla_i \lambda&=-(\tilde h^{11})^2h_{11,i},\\
\nabla_i\nabla_j \lambda&=-(\tilde h^{11})^2\nabla_i\nabla_jh_{11}+2\tilde h^{pq}\nabla_1h_{ip}\nabla_1h_{jq}.
\end{align*}
By Lemma \ref{convexity property}, we have
\begin{align*}
\partial_t \lambda &\le uF^{kl}\nabla_k\nabla_l \lambda-2uF^{-1}(\tilde h^{11})^2(\nabla_1 F)^2+\nabla \lambda \nabla\Phi F
-2\tilde h^{11}\nabla_1 \Phi \nabla_1F\\
&+(c_{n,k}\phi'+uF)-\tilde h^{11}\big((F^{kl}(h^2)_{kl}-KF^{kk})u-c_{n,k}Ku+\phi'F\big)\\
&-2KuF(\tilde h^{11})^2.
\end{align*}
On the other hand,
\begin{align*}
&\partial_t u-uF^{kl}u_{kl}\\
=&f\phi'-\nabla\Phi\nabla f-uF^{kl}u_{kl}\\
=&(c_{n,k}\phi'-uF)\phi'-\nabla\Phi\nabla (c_{n,k}\phi'-uF)-uF^{kl}(\nabla h_{kl}\nabla\Phi+\phi'h_{kl}-u(h^2)_{kl})\\
=&c_{n,k}(\phi')^2+Kc_{n,k}|\nabla \phi'|^2+\nabla \Phi\nabla uF-2u\phi'F+u^2F^{kl}(h^2)_{kl}.
\end{align*}
Note that, at $(x_0,t_0)$,
\[\nabla_i \log \lambda=\nabla_i \log u, \quad \forall i=1,\cdots,n,\]
and
\[\nabla_i\nabla_j (\log \lambda-\log u)\le 0.\]
Then we have
\begin{align*}
\frac{1}{\lambda u}(u^2F^{kl}\nabla_k\nabla_l \lambda-u\lambda F^{kl}\nabla_k\nabla_l u)&=uF^{kl}\nabla_k\nabla_l (\log \lambda-\log u),\\
\frac{1}{\lambda u}(uF\nabla \lambda\nabla \Phi-\lambda F\nabla \Phi \nabla u)&=0.
\end{align*}
Thus
\begin{align*}
0&\le\partial_t (\log \lambda-\log u)=\frac{\lambda_t}{\lambda}-\frac{u_t}{u}\\
&\le \frac{1}{\lambda u}\Big( u^2F^{kl}\nabla_k\nabla_l \lambda-2u^2F^{-1}(\tilde h^{11})^2(\nabla_1 F)^2+uF\nabla \lambda \nabla\Phi
-2u\tilde h^{11}\nabla_1 \Phi \nabla_1F\\
&+u(c_{n,k}\phi'+uF)-\tilde h^{11}\big((F^{kl}(h^2)_{kl}-KF^{kk})u^2-c_{n,k}Ku^2+\phi'uF\big)\\
&-2Ku^2F(\tilde h^{11})^2-\big(\lambda c_{n,k}(\phi')^2+\lambda Kc_{n,k}|\nabla \phi'|^2+\nabla \Phi \nabla u \lambda F-2u\lambda \phi' F\\
&+\lambda u^2F^{kl}(h^2)^{kl})+u\lambda F^{kl}\nabla_k\nabla_l u\big)\Big)\\
&=uF^{kl}\nabla_k\nabla_l(\log \lambda-\log u)+\frac{1}{\lambda u}\Big(\big(-2u^2F^{-1}(\tilde h^{11})^2(\nabla_1 F)^2-2u\tilde h^{11}\nabla_1 \Phi \nabla_1F\big)\\
&-2Ku^2(\tilde h^{11})^2-\tilde h^{11}\big((2F^{kl}(h^2)_{kl}-KF^{kk})u^2-c_{n,k}Ku^2-\phi'uF-c_{n,k}(\phi')^2\\
&+Kc_{n,k}|\nabla \phi'|^2 \big)+u(c_{n,k}\phi'+uF)\Big)\\
&=uF^{kl}\nabla_k\nabla_l(\log \lambda-\log u)+\frac{1}{\lambda u}\Big(\frac{1}{2}F(\nabla_1 \Phi)^2-2Ku^2(\tilde h^{11})^2\\
&-\tilde h^{11}\big((2F^{kl}(h^2)_{kl}-KF^{kk})u^2-c_{n,k}Ku^2-\phi'uF-c_{n,k}(\phi')^2+Kc_{n,k}|\nabla \phi'|^2 \big)\\
&+u(c_{n,k}\phi'+uF)\Big).
\end{align*}

Now we assume $K=1.$
It implies that
\begin{align*}
&-2u^2(\tilde h^{11})^2+\tilde h^{11}(F^{kk}u^2+c_{n,k}(u^2+(\phi')^2-|\nabla \phi'|^2)+u\phi' F)\\
&+u(c_{n,k}\phi'+uF)+\frac{1}{2}F(\nabla_1 \Phi)^2\ge 0.
\end{align*}
Then we have
\[(\tilde h^{11})^2-c_3\tilde h^{11}-c_4\le 0,\]
where the positive constants $c_3,c_4$ depend on the upper and lower bound of $\rho,$ the lower bound of $u$  and the upper bound of $F$. Since $u=\frac{\phi^2}{\sqrt{\phi^2+|\nabla \rho |^2}}$( see equation (4.1)  in \cite{15}),
we can conclude that  the lower bound of $u$ depends on the gradient estimate of $\rho$ ($C^1$ estimate).

The upper bound for $\tilde {h}^{11}$ implies an upper bound for $\omega$ since $u$ is bounded from below by a positive constant. This finishes the proof of the lemma.
\end{proof}

\subsection{The relation between ${\mathcal A}_m$ and ${\mathcal A}_{-1}$} Now we will use the flow (\ref{flow-CGLS}) to obtain the relation between ${\mathcal A}_m$ and ${\mathcal A}_{-1}$. In particular, the case $m=1$ will be served as the initial condition when we prove Theorem \ref{thm:1.2} by induction for odd numbers.

\begin{proposition}\label{Proposition 3.1}
Let $M$ be a closed convex $C^2$-hypersurface in $\mathbb{S}_+^{n+1},$ then the following inequality holds,
\[\mathcal{A}_{m}(\Omega)\ge \xi_{m,-1}(\mathcal{A}_{-1}(\Omega)), \quad \forall 0\le m\le n-1,\]
where $\xi_{m,-1}$ is the unique positive function defined on $(0,s_{-1})$ such that the equality holds if and only if $M$ is a geodesic sphere.
\end{proposition}

\begin{proof}
First of all, we can assume that the hypersurface is smooth and strictly convex, since otherwise we can use convolutions as in the proof of Corollary 1.2 in \cite{11} to obtain a sequence of approximating smooth strictly convex hypersurfaces converging in $C^2$ to $M$. The inequality follows from the approximation.  We will treat the equality case $\mathcal{A}_k(\Omega) = \xi_{k,k-2}(\mathcal{A}_{k-2}(\Omega))$ for general $1\le k\le n-1$ in the proof of Theorem 1.4.

When $k=0$, $F=H$ and (\ref{flow-CGLS}) becomes
\begin{equation}\label{k=0}
X_t=(n\phi'-uH)\nu.
\end{equation}
We have the $C^0$ estimate (see Proposition 4.1 in \cite{15}), $C^1$ estimate(see Proposition 5.2 in \cite{15}) and the upper bound of $H$ (see Corollary 3.3 in \cite{15}). Then we see that the convexity is preserved by Lemma \ref{convexity preserving}. Moreover, $C^2$ estimate follows directly from the theory of quasi-linear PDE. The surfaces converge exponentially to a sphere as $t \rightarrow \infty$ in the $C^{\infty}$ topology by Theorem 1.1 in \cite{15}.  Along the flow (\ref{flow-CGLS}), we have
\begin{align*}
\frac{d}{dt}\mathcal{A}_{-1}=\int_{M(t)}(n\phi'-Hu)=0,
\end{align*}
and
\begin{align*}
\frac{d}{dt}\mathcal{A}_{m}&=(m+1)\int_{M(t)}(n\phi'-Hu)\sigma_{m+1}d\mu\\
&\le (m+1)\int_{M(t)}(n\phi'\sigma_{m+1}-\frac{n(m+2)}{n-(m+1)}u\sigma_{m+2})d\mu\\
&=0,
\end{align*}
where we have used the Newton-Maclaurin inequality (Lemma \ref{lemma-NM}) and Minkowski identity (Proposition \ref{prop-Min}).
Then we have
\[\mathcal{A}_m(\Omega)\ge \xi_{m,-1}(\mathcal{A}_{-1}(\Omega)) \qquad \text{ for} \quad m=1,2,\cdots,n-1. \]
the equality holds only if $M$ is a geodesic sphere. By the definition of $\xi_{m,-1}$, we know that the equality holds if $M$ is a geodesic sphere.
Especially,
\[\mathcal{A}_1(\Omega)\ge \xi_{1,-1}(\mathcal{A}_{-1}(\Omega)),\]
with the equality holds if and if $M$ is a geodesic sphere.
\end{proof}

\section{Gerhardt's flow and applications}

In this section, following \cite{5}, we will use Gerhardt's flow to prove the main theorem.
Gerhardt \cite{14} considered the inverse curvature flows of strictly convex hypersurfaces in $\mathbb{S}^{n+1}$ and obtained smooth convergence of the flows to the equator.

Assume $F=\frac{\sigma_k}{\sigma_{k-1}}$. Then under the inverse curvature flow,
\begin{align}\label{e-ICF}
    X_t=\frac{\nu}{F}
\end{align}
with $M(0)=\partial B_{\rho}(o),$ the geodesic spheres stay as geodesic spheres and for time $t>0$, $M(t)=\partial B_{\rho(t)}(o).$ We first establish a Minkowski type inequality in $\mathbb{S}^{n+1}$ without weighted factor.

\begin{theorem}
Let $M$ be a closed convex $C^2$-hypersurface in $\mathbb{S}_+^{n+1},$ then the following inequality holds,
\[(\int_{M}\sigma_1d\mu_g)^2\ge \xi(\mathcal{A}_0^2(\Omega)),\]
where $\xi$ is the unique positive function defined on $(0,s^2_{0})$ such that the equality holds if and only if $M$ is a geodesic sphere.
Moreover, $\xi$ satisfies that
\begin{equation}\label{xi}
2\frac{n-1}{n}\xi(s)=(2n+2\xi'(s))s \quad \text{for} \quad s\in (0,s^2_{0}).
\end{equation}
In fact we can express $\xi$ explicitly by
\begin{equation}\label{def-xi-2}
\xi(s)=n^2(n+1)^{\frac{2}{n}}\omega_{n+1}^{\frac{2}{n}}s^{\frac{n-1}{n}}-n^2s,
\end{equation}
where $\omega_{n+1}$ is the volume of unit ball in ${\mathbb R}^{n+1}$.
\end{theorem}

\begin{proof}
Again we assume $M$ to be smooth and strictly convex and we use the same method as we used in the proof of Proposition \ref{Proposition 3.1}.

The function $\xi$ is defined using the following relation:
\begin{equation}\label{def-xi}
(\int_{\partial B_{\rho(t)}(o)}\sigma_1d\mu_g)^2= \xi(\mathcal{A}_0^2(B_{\rho(t)}(o))).
\end{equation}

Set $f=\frac{1}{H}$. Using (\ref{evolution of sigma}) for $l=1$ and (\ref{evolu equa A_k}) for $l=0$, we have
\begin{align*}
&\frac{d}{dt}\big((\int_{\partial B_{\rho(t)}(o)}\sigma_1d\mu_g)^2- \xi(\mathcal{A}_0^2(B_{\rho(t)}(o)))\big)\\
&=2\int_{\partial B_{\rho(t)}(o)}\sigma_1d\mu_g\int_{\partial B_{\rho(t)}(o)}\frac{1}{\sigma_1}(2\sigma_2-n\sigma_0)d\mu_g\\
&-2\xi'(\mathcal{A}_0^2(B_{\rho(t)}(o)))\mathcal{A}_0(B_{\rho(t)}(o))\int_{\partial B_{\rho(t)}(o)}\sigma_1\frac{1}{\sigma_1}d\mu_g\\
&=2\frac{n-1}{n}(\int_{\partial B_{\rho(t)}(o)}\sigma_1d\mu_g)^2-2n\int_{\partial B_{\rho(t)}(o)}\sigma_1d\mu_g\int_{\partial B_{\rho(t)}(o)}\frac{1}{\sigma_1}d\mu_g\\
&-2\xi'(\mathcal{A}_0^2(B_{\rho(t)}(o)))\mathcal{A}_0^2(B_{\rho(t)}(o))\\
&=2\frac{n-1}{n}(\int_{\partial B_{\rho(t)}(o)}\sigma_1d\mu_g)^2-2n\mathcal{A}_0^2(B_{\rho(t)}(o))-2\xi'(\mathcal{A}_0^2(B_{\rho(t)}(o)))\mathcal{A}_0^2(B_{\rho(t)}(o)),\\
\end{align*}
where the last step follows from the fact that

\begin{equation}
\mathcal{A}_0^2(\Omega(t)) \le \int_{M(t)}\sigma_1d\mu_g\int_{M(t)}\frac{1}{\sigma_1}d\mu_g,
\end{equation}
with the inequality strict unless $\sigma_1=H$ is constant on $M(t)$. Hence we obtain that the equality holds if and only if $M(t)$ is a geodesic sphere.\\
By (\ref{def-xi}), we have
\[2\frac{n-1}{n}\xi(\mathcal{A}_0^2(B_{\rho(t)}(o)))-2n\mathcal{A}_0^2(B_{\rho(t)}(o))-2\xi'(\mathcal{A}_0^2(B_{\rho(t)}(o)))\mathcal{A}_0^2(B_{\rho(t)}(o))=0.\]
We can obtain that
\[2\frac{n-1}{n}\xi(s)=(2n+2\xi'(s))s \quad \text{for} \quad s\in (0,s^2_0).\]
Let $M(t)$ solve the inverse curvature flow equation $X_t=\frac{1}{\sigma_1}\nu$ with initial condition $M(0)=M.$ Denote $\mathcal{A}_k(t)=\mathcal{A}_k(\Omega_t)$, where $\Omega_t$ is the domain enclosed by $M(t)$, then we have
\begin{align*}
&\frac{d}{dt}\big((\int_{M(t)}\sigma_1d\mu_g)^2- \xi(\mathcal{A}_0^2(t))\big)\\
&=2\int_{M(t)}\sigma_1d\mu_g\int_{M(t)}\frac{1}{\sigma_1}(2\sigma_2-n\sigma_0)d\mu_g-2\xi'(\mathcal{A}_0^2(t))\mathcal{A}_0(t)\int_{M(t)}\sigma_1\frac{1}{\sigma_1}d\mu_g\\
&\le 2\frac{n-1}{n}(\int_{M(t)}\sigma_1d\mu_g)^2-2n\int_{M(t)}\sigma_1d\mu_g\int_{M(t)}\frac{1}{\sigma_1}d\mu_g-2\xi'(\mathcal{A}_0^2(t))\mathcal{A}_0^2(t)\\
&\le 2\frac{n-1}{n}(\int_{M(t)}\sigma_1d\mu_g)^2-2n\mathcal{A}_0^2(t)-2\xi'(\mathcal{A}_0^2(t))\mathcal{A}_0^2(t)\\
&= 2\frac{n-1}{n}\big((\int_{M(t)}\sigma_1d\mu_g)^2-\xi(\mathcal{A}_0^2(t))\big).\\
\end{align*}
We will have
\[\frac{d}{dt}\Big( e^{-2\frac{n-1}{n}t}\big((\int_{M(t)}\sigma_1d\mu_g)^2- \xi(\mathcal{A}_0^2(t)\big)\Big)\le 0.\]
Denote
\[Q(t)=e^{-2\frac{n-1}{n}t}\big((\int_{M(t)}\sigma_1d\mu_g)^2- \xi(\mathcal{A}_0^2(t))\big),\]
then
\[\frac{d}{dt}Q(t)\le 0.\]
Thus
\[Q(t)-Q(0)\le 0,\]
for all $t\in [0,T^*).$\\
It is proved in Theorem 1.1 in (\cite{14}) that the curvature flow converges to an equator, as $t\rightarrow T^*,$ and with
\[|\frac{\pi}{2}-\rho|_{m,\mathbb{S}^n}\le c_m\Theta \quad \forall t\in [t_{\delta},T^*),\]
where $\Theta=\arccos e^{t-T^*}.$
It follows that
\begin{align*}
&\Vol(\Omega_t)\rightarrow \Vol(B_{\frac{\pi}{2}}) ,\\
& \int_{M_t}d\mu_g\rightarrow |\Sigma(B_{\frac{\pi}{2}})|,\\
&\int_{M_t}\sigma_k d\mu_g\rightarrow 0 \quad \forall 1\le k\le n-1,
\end{align*}
as $t\rightarrow T^*.$ By the definition of $\mathcal{A}_k,$ we have
\[\mathcal{A}_k(t)\rightarrow \mathcal{A}_k(B_{\frac{\pi}{2}}(o))  \quad \forall  -1\le k\le n-1,\]
as $t\rightarrow T^*$, which implies that
\[\lim\limits_{t \to T^*}Q(t)=0.\]
Therefore, we have
\[(\int_{M(0)}\sigma_1d\mu_g)^2- \xi(\mathcal{A}_0^2(0))\ge 0,\]
i.e.,
\[(\int_{M}\sigma_1d\mu_g)^2- \xi(\mathcal{A}_0^2(\Omega))\ge 0,\]
with the equality holds only if $M$ is a geodesic sphere. By the definition of $\xi$, we know that the equality holds if $M$ is a geodesic sphere.

\vspace{.1in}

Finally, we will derive the explicit expression for $\xi$. First we can solve the ODE (\ref{xi}) to obtain
\begin{equation}\label{e-xi-ep}
\xi(s)=s^{\frac{n-1}{n}}\epsilon^{-\frac{n-1}{n}}\xi(\epsilon)-n^2(s-\epsilon^{\frac{1}{n}}s^{\frac{n-1}{n}}).
\end{equation}
It remains to compute the limit $\lim_{\epsilon\to 0}\epsilon^{-\frac{n-1}{n}}\xi(\epsilon)$.
For this purpose, we notice that since the metric on ${\mathbb S}^{n+1}$ is given by $ds^2=d\rho^2+\sin^2\rho dz^2$, we see that
\begin{equation}\label{e-area}
{\mathcal A}_0(B_{\rho}(o))=|\partial B_{\rho}(o)|=\sin^n\rho |{\mathbb S}^n|=(n+1)\omega_{n+1}\sin^n\rho.
\end{equation}
On the other hand, for $\partial B_{\rho}(0)\subset {\mathbb S}^{n+1}$, we have $\kappa_1=\kappa_2=\cdots \kappa_n=\frac{\cos\rho}{\sin\rho}$. Hence,
\begin{equation}\label{e-sigma1}
\int_{\partial B_{\rho}(o)}\sigma_1d\mu_g=n\frac{\cos\rho}{\sin\rho}|\partial B_{\rho}(o)|=n(n+1)\omega_{n+1}\cos\rho\sin^{n-1}\rho.
\end{equation}
Therefore, if we choose $s={\mathcal A}_0^2(B_{\rho}(o))=(n+1)^2\omega_{n+1}^2\sin^{2n}\rho$, then
\begin{align*}
\xi(s)=&\left(\int_{\partial B_{\rho}(o)}\sigma_1d\mu_g\right)^2=n^2(n+1)^2\omega^2_{n+1}\cos^2\rho\sin^{2(n-1)}\rho\\
=&n^2(n+1)^2\omega^2_{n+1}\cos^2\rho\left(\frac{s}{(n+1)^2\omega_{n+1}^2}\right)^{\frac{n-1}{n}}
=n^2(n+1)^{\frac{2}{n}}\omega_{n+1}^{\frac{2}{n}}\cos^2\rho s^{\frac{n-1}{n}}.
\end{align*}
Since $\rho\to 0$ as $s\to 0$, we see that
\begin{equation*}
\lim_{\epsilon\to 0}\epsilon^{-\frac{n-1}{n}}\xi(\epsilon)=\lim_{\rho\to0}\left(n^2(n+1)^{\frac{2}{n}}\omega_{n+1}^{\frac{2}{n}}\cos^2\rho\right)=n^2(n+1)^{\frac{2}{n}}\omega_{n+1}^{\frac{2}{n}}.
\end{equation*}
Now (\ref{def-xi-2}) follows by letting $\epsilon\to 0$ in (\ref{e-xi-ep}).
\end{proof}

Now we will establish the relation between $\mathcal{A}_2$ and $\mathcal{A}_0.$

\begin{proposition}\label{Proposition 4.1}
Let $M$ be a closed convex $C^2$-hypersurface in $\mathbb{S}_+^{n+1},$ then the following inequality holds,
\[\mathcal{A}_{2}(\Omega)\ge \xi_{2,0}(\mathcal{A}_{0}(\Omega)),\]
where $\xi_{2,0}$ is the unique positive function defined on $(0,s_{0})$ such that the equality holds if and only if $M$ is a geodesic sphere.
Moreover, $\xi_{2,0}$ satisfies that
\begin{equation}\label{xi_{2,0}}
\xi'_{2,0}(s)=\frac{(n-2)\xi_{2,0}(s)-(n-1)s}{ns} .
\end{equation}
In fact we can express $\xi_{2,0}$ explicitly by
\begin{equation}\label{def-xi-2-0-2}
\xi_{2,0}(s)=\frac{n(n-1)}{2}(n+1)^{\frac{2}{n}}\omega_{n+1}^{\frac{2}{n}}s^{\frac{n-2}{n}}-\frac{n-1}{2}s,
\end{equation}
where $\omega_{n+1}$ is the volume of unit ball in ${\mathbb R}^{n+1}$.\end{proposition}

\begin{proof}
Again we assume $M$ to be smooth and strictly convex and we use the same method as we used in the proof of Proposition \ref{Proposition 3.1}.\\
For $0< \rho(t)< \frac{\pi}{2},$ the function $\xi_{2,0}$ is defined by
\begin{equation}\label{def-xi-0}
\mathcal{A}_2(B_{\rho(t)}(o))-\xi_{2,0}(\mathcal{A}_{0}(B_{\rho(t)}(o)))=0.
\end{equation}
Using (\ref{evolu equa A_k})  for $l=0$ and $l=2$ and considering the flow (\ref{e-ICF}) with $F=\frac{\sigma_2}{\sigma_1}$, we have
\begin{align*}
&\frac{d}{dt} \big( \mathcal{A}_2(B_{\rho(t)}(o))-\xi_{2,0}(\mathcal{A}_{0}(B_{\rho(t)}(o))\big)\\
&=3\int_{\partial B_{\rho(t)}(o)}\sigma_3\frac{1}{\frac{\sigma_2}{\sigma_1}}d\mu_g-\xi_{2,0}'(\mathcal{A}_0(B_{\rho(t)}(o)))\int_{\partial B_{\rho(t)}(o)}\sigma_1\frac{1}{\frac{\sigma_2}{\sigma_1}}d\mu_g\\
&=\frac{2(n-2)}{n-1}\int_{\partial B_{\rho(t)}(o)}\sigma_2d\mu_g-\frac{2n}{n-1}\xi_{2,0}'(\mathcal{A}_0(B_{\rho(t)}(o)))\int_{\partial B_{\rho(t)}(o)}\sigma_0d\mu_g,\\
\end{align*}
where the last step follows from the fact that the geodesic sphere is totally umbilic. (\ref{def-xi-0}) yields that
\begin{equation}\label{e-4.4}
 \frac{2(n-2)}{n-1}\int_{\partial B_{\rho(t)}(o)}\sigma_2d\mu_g=\frac{2n}{n-1}\xi_{2,0}'(\mathcal{A}_0(B_{\rho(t)}(o)))\int_{\partial B_{\rho(t)}(o)}\sigma_0d\mu_g.
\end{equation}
From the definition of $\mathcal{A}_k$, we have
\begin{equation}\label{e-4.5}
\mathcal{A}_2(B_{\rho(t)}(o))=\int_{\partial B_{\rho(t)}(o)}\sigma_2d\mu_g+\frac{K(n-1)}{2-1}\mathcal{A}_0(B_{\rho(t)}(o)).
\end{equation}
Combining (\ref{def-xi-0}), (\ref{e-4.4}) and (\ref{e-4.5}), we have
\[ \xi'_{2,0}(s)=\frac{(n-2)(\xi_{2,0}(s)-(n-1)s)}{ns}.\]
Let $M(t)$ solve the inverse curvature flow equation $X_t=\frac{\sigma_1}{\sigma_2}\nu$ with initial condition $M(0)=M.$ Denote $\mathcal{A}_k(t)=\mathcal{A}_k(\Omega_t)$, where $\Omega_t$ is the domain enclosed by $M(t)$, then we have
\begin{align*}
&\frac{d}{dt} \big( \mathcal{A}_2(t)-\xi_{2,0}(\mathcal{A}_{0}(t))\big)\\
&=3\int_{M(t)}\sigma_3\frac{1}{\frac{\sigma_2}{\sigma_1}}d\mu_g-\xi_{2,0}'(\mathcal{A}_0(t))\int_{M(t)}\sigma_1\frac{1}{\frac{\sigma_2}{\sigma_1}}d\mu_g\\
&\le \frac{2(n-2)}{n-1}\int_{M(t)}\sigma_2d\mu_g-\frac{2n}{n-1}\xi_{2,0}'(\mathcal{A}_0(t))\int_{M(t)}\sigma_0d\mu_g\\
&=\frac{2(n-2)}{n-1}(\mathcal{A}_2(t)-\frac{K(n-1)}{2-1}\mathcal{A}_0(t))-\frac{2n}{n-1}\xi_{2,0}'(\mathcal{A}_0(t))\mathcal{A}_0(t)\\
&=\frac{2(n-2)}{n-1}(\mathcal{A}_2(t)-\frac{K(n-1)}{2-1}\mathcal{A}_0(t))\\
&-\frac{2n}{n-1}\frac{(n-2)\xi_{2,0}(\mathcal{A}_0(t))-(n-1)\mathcal{A}_0(t)}{n\mathcal{A}_0(t)}\mathcal{A}_0(t)\\
&=\frac{2(n-2)}{n-1}(\mathcal{A}_2(t)-\xi_{2,0}(\mathcal{A}_0(t))).
\end{align*}
We have
\[\frac{d}{dt}( e^{-\frac{2(n-2)}{n-1}t}(\mathcal{A}_2(t)-\xi_{2,0}(\mathcal{A}_0(t))))\le 0.\]
Denote
\[Q_2(t)= e^{-\frac{2(n-2)}{n-1}t}(\mathcal{A}_2(t)-\xi_{2,0}(\mathcal{A}_0(t))),\]
then
\[\frac{d}{dt}Q_2(t)\le 0.\]
Thus
\[Q_2(t)-Q_2(0)\le 0,\]
for all $t\in [0,T^*).$\\
By the definition of $\mathcal{A}_k,$ we have
\[\lim\limits_{t \to T^*}Q(t)=0.\]
Therefore, we have
\[\mathcal{A}_2(0)-\xi_{2,0}(\mathcal{A}_0(0))\ge 0 ,\]
i.e.,
\[\mathcal{A}_2(\Omega)\ge \xi_{2,0}(\mathcal{A}_0(\Omega)) ,\]
 the equality holds only if $M$ is a geodesic sphere. By the definition of $\xi_{2,0}$, we know that the equality holds if $M$ is a geodesic sphere.

\vspace{.1in}
Finally, we will derive the explicit expression for $\xi_{2,0}$. First we can solve the ODE (\ref{xi_{2,0}}) to obtain
\begin{equation}\label{e-xi-ep-2}
\xi_{2,0}(s)=s^{\frac{n-2}{n}}\epsilon^{-\frac{n-2}{n}}\xi_{2,0}(\epsilon)-\frac{n-1}{2}(s-\epsilon^{\frac{2}{n}}s^{\frac{n-2}{n}}).
\end{equation}
It remains to compute the limit $\lim_{\epsilon\to 0}\epsilon^{-\frac{n-2}{n}}\xi_{2,0}(\epsilon)$.
For this purpose, recall that the area of the geodesic sphere is given by (\ref{e-area}).
On the other hand, for $\partial B_{\rho}(0)\subset {\mathbb S}^{n+1}$, we have $\kappa_1=\kappa_2=\cdots \kappa_n=\frac{\cos\rho}{\sin\rho}$. Hence, we have $\sigma_2=\frac{n(n-1)}{2}\frac{\cos^2\rho}{\sin^2\rho}$ so that
\begin{align*}
\mathcal{A}_2(B_{\rho}(o))
=&\int_{\partial B_{\rho}(o)}\sigma_2d\mu_g+(n-1)\mathcal{A}_0(B_{\rho}(o))\\
=&\frac{n(n-1)}{2}\frac{\cos^2\rho}{\sin^2\rho}|\partial B_{\rho}(o)|+(n-1)|\partial B_{\rho}(o)|.
\end{align*}
Therefore, if we choose $s={\mathcal A}_0(B_{\rho}(o))=|\partial B_{\rho}(o)|=(n+1)\omega_{n+1}\sin^{n}\rho$, then
\begin{align*}
\xi_{2,0}(s)=&\mathcal{A}_2(B_{\rho}(o))=\frac{n(n-1)}{2}\frac{\cos^2\rho}{\sin^2\rho}s+(n-1)s\\
=&\frac{n(n-1)}{2}\cos^2\rho s\left(\frac{s}{(n+1)\omega_{n+1}}\right)^{-\frac{2}{n}}+(n-1)s\\
=&\frac{n(n-1)}{2}(n+1)^{\frac{2}{n}}\omega_{n+1}^{\frac{2}{n}}\cos^2\rho s^{\frac{n-2}{n}}+(n-1)s.
\end{align*}
Since $\rho\to 0$ as $s\to 0$, we see that
\begin{align*}
\lim_{\epsilon\to 0}\epsilon^{-\frac{n-2}{n}}\xi_{2,0}(\epsilon)
=&\lim_{\rho\to0}\left(\frac{n(n-1)}{2}(n+1)^{\frac{2}{n}}\omega_{n+1}^{\frac{2}{n}}\cos^2\rho+(n-1)\epsilon^{\frac{2}{n}}\right)\\
=&\frac{n(n-1)}{2}(n+1)^{\frac{2}{n}}\omega_{n+1}^{\frac{2}{n}}.
\end{align*}
Now (\ref{def-xi-2-0-2}) follows by letting $\epsilon\to 0$ in (\ref{e-xi-ep-2}).\end{proof}

\begin{remark}
We can compare Theorem 4.1 and Proposition 4.2 with Theorem 1.5 (see (1.5) and (1.6)) of \cite{11}. The method we used to deduce Theorem 4.1 and Proposition 4.2 can be used to derive the relation between $\mathcal{A}_k$ and $\mathcal{A}_{k-2}$ for general $1\le k\le n-1.$
\end{remark}

Next, we will study the properties of the function $\xi_{k,k-2}$, which will be used in the proof of the main theorem.
\begin{proposition}\label{prop4.3}
For any $s\in (0,s_{k-2}),$ the following holds
\begin{equation}\label{xi_{k,k-2}}
\xi'_{k,k-2}(s)=\frac{n-k}{n-k+2}\frac{\xi_{k,k-2}(s)-K\frac{n-k+1}{k-1}s}{s-K\frac{n-k+3}{k-3}\xi^{-1}_{k-2,k-4}(s)} \quad \text{for}\quad k\ge 3,
\end{equation}
where $\xi_{k,k-2}$ and $\xi_{k-2,k-4}$ are defined as in (\ref{xi_{k,l}}).
\end{proposition}

\begin{proof}
For $k\ge 3$ and $0< \rho(t)< \frac{\pi}{2},$ by the definition of $\xi_{k,k-2}$, we have
\begin{equation}\label{def-xi-k}
\mathcal{A}_k(B_{\rho(t)}(o))-\xi_{k,k-2}(\mathcal{A}_{k-2}(B_{\rho(t)}(o)))=0.
\end{equation}
By (\ref{evolu equa A_k}), we have along the flow (\ref{e-ICF}) with $F=\frac{\sigma_k}{\sigma_{k-1}}$
\begin{align*}
&\frac{d}{dt} \big( \mathcal{A}_k(B_{\rho(t)}(o))-\xi_{k,k-2}(\mathcal{A}_{k-2}(B_{\rho(t)}(o))\big)\\
&=(k+1)\int_{\partial B_{\rho(t)}(o)}\sigma_{k+1}\frac{1}{\frac{\sigma_k}{\sigma_{k-1}}}d\mu_g\\
&-(k-1)\xi'_{k,k-2}(\mathcal{A}_{k-2}(B_{\rho(t)}(o)))\int_{\partial B_{\rho(t)}(o)}\sigma_{k-1}\frac{1}{\frac{\sigma_k}{\sigma_{k-1}}}d\mu_g\\
&=\frac{k(n-k)}{n-k+1}\int_{\partial B_{\rho(t)}(o)}\sigma_kd\mu_g\\
&-\frac{(n-k+2)k}{n-k+1}\xi'_{k,k-2}(\mathcal{A}_{k-2}(B_{\rho(t)}(o)))\int_{\partial B_{\rho(t)}(o)}\sigma_{k-2}d\mu_g,\\
\end{align*}
where the last step follows from the fact that the geodesic sphere is totally umbilic. (\ref{def-xi-k}) yields that
\[\frac{k(n-k)}{n-k+1}\int_{\partial B_{\rho(t)}(o)}\sigma_kd\mu_g=\frac{(n-k+2)k}{n-k+1}\xi'_{k,k-2}(\mathcal{A}_{k-2}(B_{\rho(t)}(o)))\int_{\partial B_{\rho(t)}(o)}\sigma_{k-2}d\mu_g.\]
From the definition of $\mathcal{A}_k,$ we have
\[\mathcal{A}_k({B_{\rho(t)}(o)})=\int_{\partial B_{\rho(t)}(o)}\sigma_kd\mu_g+K\frac{n-k+1}{k-1}\mathcal{A}_{k-2}({B_{\rho(t)}(o)})\]
\[\mathcal{A}_{k-2}({B_{\rho(t)}(o)})=\int_{\partial B_{\rho(t)}(o)}\sigma_{k-2}d\mu_g+K\frac{n-k+3}{k-3}\mathcal{A}_{k-4}({B_{\rho(t)}(o)})\]
Under the assumption that
\[\mathcal{A}_{k-2}(B_{\rho}(o))=\xi_{k-2,k-4}(\mathcal{A}_{k-4}(B_{\rho}(o))),\]
we have
\begin{align*}
&\frac{k(n-k)}{n-k+1}\big(\xi_{k,k-2}(\mathcal{A}_{k-2}({B_{\rho(t)}(o)})\big)-K\frac{n-k+1}{k-1}\mathcal{A}_{k-2}({B_{\rho(t)}(o)})\big)\\
&-\frac{(n-k+2)k}{n-k+1}\big(\mathcal{A}_{k-2}({B_{\rho(t)}(o)})-\frac{K(n-k+3)}{k-3}\xi^{-1}_{k-2,k-4}(\mathcal{A}_{k-2}({B_{\rho(t)}(o)}))\big)\cdot \\
&\xi'_{k,k-2}(\mathcal{A}_{k-2}({B_{\rho(t)}(o)}))=0.
\end{align*}
Then we have
\[\xi'_{k,k-2}(s)=\frac{n-k}{n-k+2}\frac{\xi_{k,k-2}(s)-K\frac{n-k+1}{k-1}s}{s-K\frac{n-k+3}{k-3}\xi^{-1}_{k-2,k-4}(s)}\]
for any $s\in (0,s_{k-2}).$
\end{proof}

Now we can prove our main result.

\begin{proof}[Proof of Theorem \ref{thm:1.2}]
Case 1. $M$ is a closed strictly convex and smooth hypersurface. We will prove Theorem 1.4 by reduction.

Let $M(t)$ solve the inverse curvature flow equation $X_t=\frac{\sigma_{k-1}}{\sigma_k}\nu$ with initial condition $M(0)=M.$ Denote $\mathcal{A}_k(t)=\mathcal{A}_k(\Omega_t)$, where $\Omega_t$ is the domain enclosed by $M(t)$.
For $k=1,2$, the inequality (\ref{def-Ak-Ak-2}) holds by Proposition \ref{Proposition 3.1} with $m=1$ and Proposition \ref{Proposition 4.1}, respectively.
Then we can assume that
\[\mathcal{A}_{k-2}(\Omega)\ge \xi_{k-2,k-4}(\mathcal{A}_{k-4}(\Omega)),\]
for any strictly convex hypersurface $M$ in $\mathbb{S}^{n+1}$ and the equality holds if and only if $M$ is a geodesic sphere.
Since
\[\frac{d}{dt} ( \mathcal{A}_{k-2}(t))=(k-1)\int_{M(t)}\sigma_{k-1}\frac{1}{\frac{\sigma_k}{\sigma_{k-1}}}d\mu_g\ge \frac{k(n-k+2)}{n-k+1} \int_{M(t)}\sigma_{k-2}d\mu_g>0\] and $M(t)$ converges to an equator of $\mathbb{S}^{n+1}$, we have
\[\mathcal{A}_{k-2}(t)\in (0,s_{k}).\]
By Proposition \ref{prop4.3} and applying Newton-Maclaurin inequality, we have
\begin{align*}
&\frac{d}{dt} \big( \mathcal{A}_k(t)-\xi_{k,k-2}(\mathcal{A}_{k-2}(t))\big)\\
=&(k+1)\int_{M(t)}\sigma_{k+1}\frac{1}{\frac{\sigma_k}{\sigma_{k-1}}}d\mu_g-(k-1)\xi'_{k,k-2}(\mathcal{A}_{k-2}(t))\int_{M(t)}\sigma_{k-1}\frac{1}{\frac{\sigma_k}{\sigma_{k-1}}}d\mu_g\\
\le&\frac{k(n-k)}{n-k+1}\int_{M(t)}\sigma_kd\mu_g-\frac{(n-k+2)k}{n-k+1}\xi'_{k,k-2}(\mathcal{A}_{k-2}(t))\int_{M(t)}\sigma_{k-2}d\mu_g\\
=&\frac{k(n-k)}{n-k+1}(\mathcal{A}_k(t)-K\frac{n-k+1}{k-1}\mathcal{A}_{k-2}(t))\\
&-\frac{(n-k+2)k}{n-k+1}\xi'_{k,k-2}(\mathcal{A}_{k-2}(t))(\mathcal{A}_{k-2}(t)-K\frac{n-k+3}{k-3}\mathcal{A}_{k-4}(t))\\
=&\frac{k(n-k)}{n-k+1}(\mathcal{A}_k(t)-K\frac{n-k+1}{k-1}\mathcal{A}_{k-2}(t))\\
&-\frac{(n-k+2)k}{n-k+1}\frac{n-k}{n-k+2}\frac{\xi_{k,k-2}(\mathcal{A}_{k-2}(t))-K\frac{n-k+1}{k-1}\mathcal{A}_{k-2}(t)}{\mathcal{A}_{k-2}(t)-K\frac{n-k+3}{k-3}\xi^{-1}_{k-2,k-4}(\mathcal{A}_{k-2}(t))}\cdot\\
&(\mathcal{A}_{k-2}(t)-K\frac{n-k+3}{k-3}\mathcal{A}_{k-4}(t))\\
\le& \frac{k(n-k)}{n-k+1}(\mathcal{A}_k(t)-K\frac{n-k+1}{k-1}\mathcal{A}_{k-2}(t))\\
&-\frac{k(n-k)}{n-k+1}\left(\xi_{k,k-2}(\mathcal{A}_{k-2}(t))-K\frac{n-k+1}{k-1}\mathcal{A}_{k-2}(t)\right),
\end{align*}
where in the last step we use the assumption that
\[ \mathcal{A}_{k-4}(t)\le \xi^{-1}_{k-2,k-4}(\mathcal{A}_{k-2}(t)).\]
Therefore we have
\[\frac{d}{dt} \big( \mathcal{A}_k(t)-\xi_{k,k-2}(\mathcal{A}_{k-2}(t)\big)\le \frac{k(n-k)}{n-k+1}(\mathcal{A}_k(t)-\xi_{k,k-2}(\mathcal{A}_{k-2}(t))).\]
Assume that \[Q_k(t)=e^{-\frac{k(n-k)}{n-k+1}t}(\mathcal{A}_k(t)-\xi_{k,k-2}(\mathcal{A}_{k-2}(t))),\]
then
\[\frac{d}{dt}Q_k(t)\le 0.\]
Thus
\[Q_{k}(t)-Q_k(0)\le 0,\]
for all $t\in [0,T^*).$\\
By the definition of $\mathcal{A}_k,$ we have
\[\lim\limits_{t \to T^*}Q_k(t)=0.\]
Therefore,
\[\mathcal{A}_k(0)-\xi_{k,k-2}(\mathcal{A}_{k-2}(0))\ge 0 ,\]
i.e.,
\[\mathcal{A}_k(\Omega)\ge \xi_{k,k-2}(\mathcal{A}_{k-2}(\Omega)),\]
 the equality holds only if $M$ is a geodesic sphere. By the definition of $\xi_{k,k-2}$, we know that the equality holds if $M$ is a geodesic sphere.

Case 2.  $M$ is a closed convex $C^2$-hypersurface.

We can obtain a sequence of approximating smooth strictly convex hypersurfaces converging in $C^2$ to $M$.
The inequality follows from the approximation.
 We now treat the equality case $\mathcal{A}_k(\Omega) = \xi_{k,k-2}(\mathcal{A}_{k-2}(\Omega))$ for general $1\le k\le n-1$.

Assume $M$ is convex and the equality holds for general $1\le k \le n-1$, we will show $M$ is strictly $k$-convex. To see this, note that both $\mathcal{A}_{k}$ and $\mathcal{A}_{k-2}$ are positive, since there exists at least one elliptic point on a closed hypersurface in $\mathbb{S}_+^{n+1}.$ Let $M_+=\{x\in M|\sigma_k(x)>0\}.$ $M_+$ is open and nonempty. We claim that $M_+$ is closed. This would imply $M=M_+,$ so $M$ is strictly $k$ convex.

We now prove that $M_+$ is closed. We will follow the idea of \cite{3}. Pick any $\eta \in C^2_0(M_+)$ compactly supported in $M_+$. Let $M_s$ be the hypersurface determined by position function $X_s=X+s\eta \nu,$ where $X$ is the position function  of $M$ and $\nu$ is the unit outernormal of $M$ at $X$. Let $\Omega_s$ be the domain enclosed by $M_s$. It is easy to show that $M_s$ is $k$-convex when $s$ is small enough.
Define
\[\mathcal{I}_k(\Omega_s)=\mathcal{A}_k(\Omega_s)- \xi_{k,k-2}(\mathcal{A}_{k-2}(\Omega_s)).\]
Therefore $\mathcal{I}_k(\Omega_s)-\mathcal{I}_k(\Omega)\ge 0$ for $s$ small, which implies that
\[\frac{d}{ds}\mathcal{I}_k(\Omega_s)|_{s=0}=0.\]
Simple calculation yields

\[\partial_t \mathcal{A}_l=(l+1)\int_Mf\sigma_{l+1}.\]
Therefore,
\[\frac{d}{ds}\mathcal{I}_k(\Omega_s)|_{s=0}=(k+1)\int_M( \sigma_{k+1}-c_1\sigma_{k-1})\eta d\mu_g=0,\]
where $c_1=\frac{k-1}{k+1}\xi'_{k,k-2}(\mathcal{A}_{k-2})>0$ and for any $\eta \in C^2_0(M_+).$ Thus,
\begin{equation}\label{variation}
\sigma_{k+1}=c_1\sigma_{k-1}, \qquad \forall x\in M_+.
\end{equation}
It follows from the Newton-Maclaurine inequality that there is a dimensional constant $C(n,k)$ such that
\[\sigma_{k+1}\le C(n,k)\sigma^{1+\frac{2}{k-1}}_{k-1}(x), \quad x\in M_+.\]
By (\ref{variation}), there is a positive $c_2,$ such that
\[\sigma_{k-1}\ge c_2>0,\qquad \forall x\in M_+,\]
where $c_2=(\frac{c_1}{C(n,k)})^{\frac{k-1}{2}}$ is a positive constant depending only on $n,k$ and $\Omega.$
In view of  (\ref{variation}), we have
\[\sigma_{k+1}\ge c_1c_2 ,\qquad \forall x\in M_+,\]
which implies that
\[\sigma_k\ge c_3>0,\]
where $c_3=\sqrt{\frac{(n-k+1)(k+1)}{k(n-k)}c_1c_2^2}$ is a positive constant depending only on $n,k$ and $\Omega.$. It follows that $M_+$ is closed.

Then we claim that  the flow $(\ref{flow-CGLS})$ preserves the convexity in a short time. We denote $M$ by $M_0$. Approximate the initial surface $M_0$ by a strictly convex ones $M^{\epsilon}_0$, by the implicit function theorem, there exists a $t_0>0$ (independent of $\epsilon$) such that  $(\ref{flow-CGLS})$ has a regular solution $M^{\epsilon}(t)$ (with $C^0,C^1,C^2$ bounds) which satisfies that $M^{\epsilon}(0)=M^{\epsilon}_0$  for $0\le t\le t_0$. Strict convexity is preserved for the approximate flows $M^{\epsilon}(t)$ by Lemma \ref{convexity preserving}. Letting $\epsilon \rightarrow 0,$  we could obtain that $M(t)$ is convex for any $0\le t\le t_0$.

Now we will prove that the equality  $\mathcal{A}_k(\Omega_t) = \xi_{k,k-2}(\mathcal{A}_{k-2}(\Omega_t))$ would be preserved at least in a short time along the flow $(\ref{flow-CGLS})$.  On one hand, by (\ref{evolu equa A_k}), we could obtain that 
\[\frac{d}{dt}\mathcal{A}_k=(k+1)\int(c_{n,k}\phi'\sigma_{k+1}-\frac{\sigma^2_{k+1}}{\sigma_k})\le 0,\]
and 
\[\frac{d}{dt}\mathcal{A}_{k-2}=(k-1)\int(c_{n,k}\phi'\sigma_{k-1}-\frac{\sigma_{k+1}\sigma_{k-1}}{\sigma_k})\ge 0.\] 
It implies that $\mathcal{A}_k(\Omega_t)\le \mathcal{A}_k(\Omega)$ and $\mathcal{A}_{k-2}(\Omega)\le \mathcal{A}_{k-2}(\Omega_t)$ for any $0\le t\le t_0$. Since $\mathcal{A}_k(\Omega)= \xi_{k,k-2}(\mathcal{A}_{k-2}(\Omega)),$ we have
\[\mathcal{A}_k(\Omega_t)\le \mathcal{A}_k(\Omega)=\xi_{k,k-2}(\mathcal{A}_{k-2}(\Omega))\le \xi_{k,k-2}(\mathcal{A}_{k-2}(\Omega_t)),\]
 along the flow $(\ref{flow-CGLS})$. On the other hand, since the convexity is preserved, we could obtain the inequality  $\mathcal{A}_k(\Omega_t) \ge \xi_{k,k-2}(\mathcal{A}_{k-2}(\Omega_t))$.

As a result, the equality of the Newton-Maclaurin inequality must be held at every point of M(t) for each $0\le t\le t_0.$ This implies that $M(t)$ is a geodesic sphere for each $0\le t\le t_0.$ In particular, $M_0$ is a geodesic sphere.

This finishes the proof of the theorem.
 \end{proof}

\renewcommand{\abstractname}{Acknowledgements}
\begin{abstract}
The authors would like to thank Professor Jiayu Li and Professor Pengfei Guan for helpful suggestions. The authors would also be grateful for useful discussions with Professor Yong Wei.
\end{abstract}
 ~\\



\begin{thebibliography}{99}
\bibitem{1}
A.D. Alexandrov, {\em Zur Theorie der gemischten Volumina von konexen K$\ddot{o}$rpern, $\Rmnum{2}$. Neue Ungleichungen zwischen den gemischten Volumina und ihre Anwendungen}, Mat. Sb.(N.S.), {\bf 2} (1937), 1205-1238 (in Russian).
\bibitem{2}
A.D. Alexandrov,  {\em Zur Theorie der gemischten Volumina von konexen K$\ddot{o}$rpern, $\Rmnum{3}$. Die Erweiterung zweeier Lehrsatze Minkowski $\ddot{u}$ber die konvexen Polyeder auf beliebige konvexe Flachen}, Mat. Sb.(N.S.), {\bf 3} (1938), 27-46 (in Russian).
\bibitem{21}
B. Andrews, X. Chen, Y. Wei, {\em Volume preserving flow and Alexandrov-Fenchel type inequalities in hyperbolic space},  J. Euro. Math. Soc., arXiv: 1805.11776v1 (to appear).
\bibitem{5}
S. Brendle, P. Guan and J. Li, {\em An inverse curvature type hypersurface flow in space forms}, preprint.
\bibitem{4}
S. Brendle,  P. Hung and M. Wang, {\em A Minkowski inequality for hypersurfaces in the anti-de Sitter-Schwarzschild manifold}, Comm. Pure Appl. Math., {\bf 69} (2016), no. 1, 124-144.
\bibitem{13}
C. Chen, P. Guan, J. Li and J. Scheuer, {\em A curvature hypersurface flow in $\mathbb{S}^{n+1}$}, work in progress.
\bibitem{6}
Y. Ge, G. Wang and J. Wu,  {\em Hyperbolic Alexandrov-Fenchel quermassintegral inequalities}, $\Rmnum{2}$, J. Diff. Geom., {\bf 98} (2014), no. 2, 69-96.
\bibitem{14}
C. Gerhardt, {\em Curvature flows in the sphere}, J. Differ. Geom., {\bf 100} (2014), 301-347.
\bibitem{10}
F. Gir$\tilde{a}o$ and N. Pinheiro, {\em An Alexandrov-Fenchel type inequality for hypersurfaces in the sphere}, Ann. Glob.Anal.Geom., {\bf 52} (2017), 413-424.
\bibitem{18}
P. Guan, {\em Curvature measures, isoperimetric inequalities and fully nonlinear PDEs}, "Fully Nonlinear PDEs in Real and Complex Geometry and Optics", Cetraro, Italy 2012, Editors: Cristian E.Gutiarrez, Ermannno Lanconelli, Springer(2014), 47-94.
\bibitem{3}
P. Guan and J. Li, {\em The quermassintegral inequalities for k-convex starshaped domains}, Adv. Math., {\bf 221} (2009), no. 5, 1725-1732.
\bibitem{15}
P. Guan and J. Li, {\em A mean curvature type flow in space forms}, Int. Math. Res. Not., (2015), no. 13, 4716-4740.
\bibitem{16}
P. Guan and J. Li, {\em Isoperimetric type inequalities and hypersurface flows}, J. Math. Study. A special issue on the occasion of 70th birthdays of professors A. Chang and P. Yang. http:// www. math. mcgill.ca/guan/Guan-Li-2019S1.pdf.
\bibitem{8}
Y. Hu, H. Li and Y. Wei, {\em Locally constraint inverse curvature flows in hyperbolic space}, Math. Annal.,  arXiv: 2002.10643v2 (to appear).
\bibitem{9}
H. Li, Y. Wei and C. Xiong, {\em A geometric inequality on hypersurface in hyperbolic space}, Adv. Math.,  {\bf 253} (2014), 152-162.
\bibitem{11}
 M. Makowski and J. Scheuer,  {\em Rigidity results, inverse curvature flows and Alexandrov-Fenchel type inequalities in the sphere}, Asian J. Math.,  {\bf 20} (2016), no. 5, 869-892.
\bibitem{17}
J. Parbosa and A. Colares, {\em Stability of hypersurfaces with constant $r$-mean curvature}, Ann. Global Anal. Geom., {\bf 15} (1997), no. 3, 277-297.
\bibitem{20}
G. Solanes, Integral geometry and the Gauss-Bonnet theorem in constant curvature spaces, Trans. Amer. Math. Soc., {\bf 358} (2006), no. 3, 1105-1115.
\bibitem{19}
J. Urbas, {\em An expansion of convex hypersurfaces}, J.Differ. Geom., {\bf 33} (1991), no. 1, 91-125.
\bibitem{7}
G. Wang and C. Xia, {\em Isoperimetric type problems and Alexandrov-Fenchel type inequalities in the hyperbolic space}, Adv. Math., {\bf 259} (2014), 523-556.
\bibitem{12}
Y. Wei and C. Xiong, {\em Alexandrov-Fenchel type inequalities for convex hypersurfaces in hyperbolic space and in sphere}, Pac. J. Math., {\bf 277} (2015), no. 1, 219-239.




		

\end{thebibliography}
\end{document}